\documentclass[a4paper,10pt]{article}

\usepackage{amsmath,amssymb,mathtools,amsthm}
\usepackage{xcolor,graphicx,float,enumerate}
\graphicspath{{figures/}}

\usepackage[T1]{fontenc}

\usepackage[utf8]{inputenc}
\DeclareUnicodeCharacter{00A0}{ }
\usepackage{geometry}
\usepackage{enumerate}

\newcommand{\ee}{\varepsilon}
\newcommand{\defeq}{\vcentcolon=}
\newcommand{\definedby}{=\vcentcolon}

\DeclareMathOperator{\N}{{\mathbb N}}

\newcommand{\Rd}{{\mathbb R^n}}

\DeclareMathOperator{\supp}{supp}

\DeclareMathOperator*{\esssup}{ess\,sup}

\newcommand*\diff{\mathop{}\!\mathrm{d}}

\newcommand{\vertiii}[1]{{\left\vert\kern-0.25ex\left\vert\kern-0.25ex\left\vert #1
		\right\vert\kern-0.25ex\right\vert\kern-0.25ex\right\vert}}

\usepackage[hidelinks]{hyperref}
\usepackage[nameinlink]{cleveref}

\newtheorem{theorem}{Theorem}[section]

\theoremstyle{definition}
\newtheorem{remark}{Remark}%

\numberwithin{equation}{section}

\newcommand{\bc}{\color{blue}}

\newcommand{\R}{\mathbb{R}}
\newcommand{\Z}{\mathbb{Z}}
\newcommand{\G}{\mathcal{G}}
\newcommand{\veps}{\varepsilon}
\newcommand{\dd}[1]{\diff{#1}}

\definecolor{darkblue}{rgb}{0.05, .05, .65}
\definecolor{darkgreen}{rgb}{0.1, .65, .1}
\definecolor{darkred}{rgb}{0.8,0,0}

\usepackage[
url=false,
isbn=false,
backend=bibtex,
maxcitenames=4,
giveninits,
maxbibnames=100,
doi=true,
uniquename=false,
block=none,
sorting=nyt]{biblatex}
\renewbibmacro{in:}{}
\DeclareFieldFormat[article]{citetitle}{#1}
\DeclareFieldFormat[article]{title}{#1}

\makeatletter
\renewcommand*{\@fnsymbol}[1]{\ensuremath{\ifcase#1\or \star \or \dagger\or \ddagger\or
		\mathsection\or \mathparagraph\or \|\or **\or \dagger\dagger
		\or \ddagger\ddagger \else\@ctrerr\fi}}
\makeatother

\providecommand{\mkbibnamefamily}[1]{#1}
\newcommand\lastnameformat[1]{\MakeUppercase{#1}}
\AtBeginBibliography{%
	\renewcommand\mkbibnamelast{\lastnameformat}%
	\renewcommand\mkbibnamefamily{\lastnameformat}%
}
\let\lastnameformat=\textsc

\title{ Estimates on translations and Taylor expansions  \\ in fractional Sobolev spaces}
\author{Félix del Teso\thanks{Dpto.\ de Análisis Matemático y Matemática Aplicada, Universidad Complutense de Madrid. \href{mailto:felix.delteso@ucm.es}{fdelteso@ucm.es}} %
	\and David Gómez-Castro\footnotemark[1]\
	\thanks{Instituto de Matemática Interdisciplinar, Universidad Complutense de Madrid. \href{mailto:dgcastro@ucm.es}{dgcastro@ucm.es}} \and %
	Juan Luis Vázquez\thanks{Departamento de Matemáticas, Universidad Autónoma de Madrid. \href{mailto:juanluis.vazquez@uam.es}{juanluis.vazquez@uam.es}}}

\newcounter{countersaver}

\begin{document}
	\maketitle

\begin{abstract} \noindent
	In this paper we study how the (normalised) Gagliardo semi-norms $[u]_{W^{s,p} (\Rd)}$ 	control translations. 	In particular, we prove that $\| u(\cdot + y)  - u  \|_{L^p (\Rd)} \le C [ u ] _{W^{s,p} (\Rd)} |y|^s$ 	for $n\geq1$, $s \in [0,1]$ and $p \in [1,+\infty]$, where $C$ depends only on $n$. We then obtain a corresponding higher-order version of this result: we get fractional rates of the error term in the Taylor expansion. We also  present relevant implications of our two results. First, we obtain a direct proof of several compact embedding of $W^{s,p}(\R^n)$ where the Fréchet-Kolmogorov Theorem is applied with  known rates. We also derive fractional rates of convergence of the convolution of a function with suitable mollifiers. Thirdly, we obtain  fractional rates of convergence of finite-difference discretizations for $W^{s,p} (\Rd)$.

	\


	\
	
\end{abstract}

\maketitle

\section{Introduction and main results}\label{sec:intro}

It is a classical result that for every $u\in W^{1,1}(\Rd)$ the following translation estimate holds:
\begin{equation}
\label{eq:translation W1p}
	\| u(\cdot + y) - u \| _{L^p (\Rd)} \le  \| \nabla u \|_{L^p (\Rd)} |y|.
\end{equation}
This result has deep implications which are well-known. First, the compact embedding of $W^{1,p} $ in $L^p$ via the Fréchet-Kolmogorov Theorem.
A second well-known application is  the estimation of  the  remainder of the Taylor polynomial. This is a key tool used to obtain optimal orders of the convergence of convolutions as well as consistency of finite-difference discretizations. Many other applications can  also be found in the literature.

\medskip

The aim of this paper is to revisit such a theory in the context of  fractional Sobolev spaces (which are  conveniently presented e.g. in \cite{Adams+Fournier:2003sobolev+spaces,Mazya2011,Tartar2007}). It is common to define the fractional Sobolev  spaces  $W^{s,p} (\Rd)$ in the Sobolev–Slobodeckii form. Thus, for $s \in (0,1)$
and $p \in [1,\infty)$ we define the normalised Gagliardo seminorm
\begin{equation*}
[u]_{W^{s,p} (\Rd)} =  \left( s (1-s) \int_{\Rd} \int_{\Rd} \frac{|u(x) - u(y)|^p}{|x-y|^{n+sp}} \diff x \diff y \right)^{\frac 1 p}
\end{equation*}
and then
$
W^{s,p}(\Rd):= \{u\in L^p(\Rd)\, : \, [u]_{W^{s,p} (\Rd)}<+\infty \}
$, which
is a Banach space with the norm
\begin{equation*}
\| u \|_{W^{s,p} (\Rd)} = \| u \|_{L^p (\Rd)} + [u]_{W^{s,p} (\Rd)} .
\end{equation*}
The limit cases $s=0$, $s=1$ for $p\in[1,\infty]$ and $p=\infty$ for $s\in(0,1)$ are defined respectively by
\[
[u]_{W^{0,p} (\Rd)} = \|  u \|_{L^p (\Rd)}, \ [u]_{W^{1,p} (\Rd)} = \| \nabla u \|_{L^p (\Rd)}, \   \textup{and} \ [u]_{W^{s,\infty} (\Rd)} = \esssup_{x,y\in \R^n} \frac{|u(x) - u(y)|}{|x-y|^{s}}.
\]
The basic result of this paper asserts that the translation estimate holds in the  following fractional Sobolev \normalcolor setting.

\begin{theorem}
	\label{thm:main} The following estimate holds for every $u \in W^{s,p} (\Rd)$  with 	$p \in [1,+\infty]$ 	and $s \in [0,1]$:
	\begin{equation}
	\label{eq:fractional taylor less 1}
	\| u( \cdot + y )  - u  \|_{L^p (\Rd)} \le C\,   [ u ] _{W^{s,p} (\Rd)} |y|^s.
	\end{equation}
The constant $C>0$  depends only $n$, and not on $s$ or $p$. An admissible value is $C=4 n(n+1) e ^{\frac 1{e \omega_n n}}$.
\end{theorem}

Let us  briefly comment on previous results related to \Cref{thm:main}.
 This kind of translation estimate for $n = 1$ can be found in Simon's \cite[Lemma 5]{Simon1986} in an abstract setting (vector valued functions)  and without explicit bounds on the constants. Like in \cite{Simon1986},  the technique of our proof  \normalcolor  is based on $K$-interpolation theory (which we introduce in detail in \Cref{sec.thm}).  On the other hand,
the result
\begin{equation*}
	\| u( \cdot + y )  - u  \|_{L^p (\Rd)} \le s^{-\frac 1 p} C(n,p)  [ u ] _{W^{s,p} (\Rd)} |y|^s
\end{equation*}
was obtained for $p = 1$ by Ambrosio--de Philippis--Martinazzi \cite[Proposition 4]{Ambrosio2011} and for $p \in (1,\infty)$ by Brasco--Lindgren--Parini \cite[Lemma A.1]{Brasco2014} via direct computations. As we recall in \Cref{rem:normalised Gagliardo-Nirenberg}, the normalised Gagliardo seminorm converges to the integer seminorm up to a constant. Hence, \cite{Ambrosio2011,Brasco2014} do not provide uniform estimates as $s \to 0$.  However,  it seems  very natural to  include  this  limit  case since for $s= 0$ we have the simple rule $\| u(\cdot + y) - u \|_{L^p (\Rd)} \le 2 \|u \|_{L^p (\Rd)}$.
We recall that these  results can also be understood as embeddings of
Besov spaces (see a detailed comment in \Cref{rem:Tartar}).

\medskip

Compared to previous related literature, a novel contribution of \Cref{thm:main} is the fact that we obtain   a clean statement  with a constant $C$ which is uniform   on the whole ranges  $s\in[0,1]$  and $p\in[1,\infty]$ (see \Cref{rem:normalised Gagliardo-Nirenberg}). Essentially, this is done by recalling a sharp
equivalence between the Gagliardo seminorm and the norm in some interpolation spaces given in the recent paper by Brasco and Salort \cite{Brasco2019}. The control of the constants in inequalities like ours can be important in the applications. %

\medskip

Once \Cref{thm:main}  is established,
we can prove the corresponding result with higher-order Sobolev regularity.
It is well known that a function $u$ in the H\"older space $C^{k,s}(\Rd)$ for some $k\geq0$ integer  and some $s\in [0,1]$ satisfies
\[
\|u(\cdot+y)- P_ku(\cdot,y) \|_{L^\infty(\Rd)} \leq C \,[ u]_{C^{k,s}(\Rd)} |y|^{k+s}.
\]
where $C=C(k)$ is a positive constant and $P_ku$ denotes the Taylor polynomial of order $k$ and centred at $x$ of the function $u$, i.e
\[
P_ku(x,y)\defeq \sum_{|\alpha|\leq k} \frac{D^\alpha u(x)}{\alpha !} y^\alpha,
\]
 where we are using the standard  multi-index notation (see \eqref{eq:multiin} below). The higher-order fractional Sobolev norms for $k\in \mathbb N$ and $s \in
 (0,1]$ are defined  as %
\begin{align*}
\| u \|_{W^{k+s,p} (\Rd) } & \defeq \| u \|_{W^{ k , p} (\Rd)} +  [u]_{W^{k+s,p} (\Rd)}, \quad \textup{where} \quad
[u]_{W^{k+s,p} (\Rd)}  \defeq \max_{|\alpha |= k }[ D^\alpha u  ]_{W^{s ,p} (\Rd)}
\end{align*}
and $W^{k+s,p}(\Rd):= \{u\in W^{k,p}(\Rd)\, : \, [u]_{W^{k+s,p} (\Rd)}<+\infty \} $.
For $p = \infty$ they coincide with the Hölder spaces: for $k\geq0$ integer and $s \in (0,1)$ we have $W^{k+s,\infty}(\R^n)  = C^{k, s}(\R^n) $.

\medskip

In  this higher-order fractional  Sobolev setting, we prove  the following estimate of the remainder term in the Taylor expansion.

\begin{theorem}
	\label{thm:fractional taylor}
	 Let $u \in W^{ k + s , p} (\Rd) $ where 	 $p \in [1,\infty]$,
	 $s \in [0,1]$, and  $k\geq0$ is an integer. The following estimate holds
		\begin{equation*}
		\| u( \cdot + y )  - P_{k} (\cdot, y)  \|_{L^p (\Rd)} \le   		 C\, [  u ] _{W^{k+s,p} (\Rd)}
		  |y|^{k+s},
		\end{equation*}
		where $C>0$ depends only on $n$ and $k$.
\end{theorem}

\medskip

\noindent {\bf Notation.}
Throughout the paper, $\omega_n$ denotes the volume of the unit ball of $\Rd$ and
we will use the multi-index notation $\alpha=(\alpha_1,\ldots,\alpha_n) \in \mathbb N^n$, $x=(y_1,\ldots,y_n)\in \Rd$,
\begin{equation}\label{eq:multiin} \quad |\alpha|=\sum_{i=1}^n \alpha_i, \quad \alpha!= \prod_{i=1}^n \alpha_i !, \quad y^\alpha=\prod_{i=1}^n y_i^{\alpha_i}, \quad \textup{and} \quad D^\alpha u=\frac{\partial^{|\alpha|}u}{\partial x_1^{\alpha_1} \cdots x_n^{\alpha_n}}.
\end{equation}

\noindent {\bf Structure of the paper.}   In the rest of \Cref{sec:intro} we make some \normalcolor comments on relevant results that in a way or another intersect and complement ours.  In \Cref{sec.app} we present several applications of \Cref{thm:fractional taylor}. We delay the proof of the main theorems to
\Cref{sec.thm}. We conclude the paper with    a list of \normalcolor comments and extensions in \Cref{sec:comments}.

\subsection{Comments on the limits as $s\to0^+$ and $s\to1^-$}
	\label{rem:normalised Gagliardo-Nirenberg}
Let us note that
for
$p \in [1,+\infty)$ and $u \in C_c^\infty (\Rd)$
the normalised Gagliardo  seminorm converges to classical seminorms up to constants
\begin{align}
\label{eq:limit seminorm to 1}
[ u ]_{W^{s,p} (\Rd)} &\to K_1 \| \nabla u \|_{L^p (\Rd)} && \text{as } s \to 1^- \\
\label{eq:limit seminorm to 0}
[ u ]_{W^{s,p} (\Rd)} &\to K_0 \| u \|_{L^p (\Rd)} && \text{as } s \to 0^+,
\end{align}
where the different constants $K_i=K_i (n,p) $  are positive (see Bourgain-Brezis-Mironescu in \cite[Corollary 2]{Bourgain+Brezis2001} and Maz'ya-Shaposhnikova in \cite[Theorem 3]{Mazya2002}).
Hence, as $s$ approaches $1$, we recover \eqref{eq:translation W1p} up to a constant.
Dávila \cite{Davila2002} showed a similar result for $u \in BV (\Omega)$.
This kind of interpolation norm convergence at the endpoints of the interpolation is discussed in a general setting in \cite{Milman2005}.
\subsection{Related known results in Besov spaces}
	\label{rem:Tartar}
	The Besov spaces $B^s_{p,\infty} (\Rd)$ can be defined as the functions in $L^p(\Rd)$ such that $\| u(\cdot + y) - u \|_{L^p} \le C |y|^s$ for some constant $C$. Hence, \Cref{thm:main} implies, in particular, that
	$W^{s,p}(\Rd) \subset B^{s}_{p,\infty}(\Rd)$. 	
	
	There are two equivalent ways of defining the Besov spaces $B^s_{p,q}(\R^n)$. The first one is in terms of the integrability of the $p$-modulus of continuity as  presented in \cite{Bennet1988}. More precisely, for $s \in (0,1]$, we take  $\omega_{p} (u,t) = \sup_{|y| \le t}  \| u(\cdot + y) - u \|_{L^p(\R^n)}$ and define the Besov space $B^s_{p,q}(\R^n)$ as the functions in $L^p(\R^n)$ with finite Besov norm
	\begin{equation*}
	\| u \|_{B^s_{p,q} (\Rd)} = \begin{dcases}
	 \|u \|_{L^p (\Rd)}+ \left( \int_0^\infty (t^{-s} \omega_{p} (u,t) )^q \frac{dt}t \right)^{\frac 1 q} &  \quad \textup{for} \quad  q <\infty , \\
	 \|u \|_{L^p (\Rd)}+ \sup_{t>0} \left( t^{-s} \omega_{p} (u,t)\right)  & \quad \textup{for} \quad q = \infty.
	\end{dcases}
	\end{equation*}
	This norm is fairly similar to
	the Gagliardo norm in fractional Sobolev spaces,
	and in fact it is known that
	 $W^{s,p}(\R^n) = B^s_{p,p}(\R^n)$. Note also that a version of   \eqref{eq:fractional taylor less 1} \normalcolor can be written in terms of $p$-modulus as
	\begin{equation*}
		\sup_{t > 0} (t^{-s} \omega_{p} (u,t)) \le C [u]_{W^{s,p}(\R^n)}.
	\end{equation*}
	This definition of Besov spaces when $s >1$ requires the higher order $p$-modulus of continuity, for which we refer the reader to \cite[Section 4 in Chapter 5]{Bennet1988}.
	
	The second definition (see \cite{Adams+Fournier:2003sobolev+spaces,Tartar2007}),  valid for $q<+\infty$, \normalcolor is given in terms of $K$-interpolation (that will be properly introduced in \Cref{sec.thm}) as
	\begin{equation*}
		B^{s}_{p,q} (\Rd) = (L^p (\Rd), W^{k,p} (\Rd))_{\frac s k, q; K}
	\end{equation*}
	where $k$ is a positive  integer greater than $s$.
	The equivalence of both definitions  is proven in any of the references (see, e.g.,
	\cite[Theorem 7.47]{Adams+Fournier:2003sobolev+spaces}
	or \cite[Chapter 35]{Tartar2007}).

	In fact, for the interpolation is known (see \cite[Corollary 7.17]{Adams+Fournier:2003sobolev+spaces}) that $(X_0, X_1)_{\theta,p;K} \subset (X_0,X_1)_{\theta, q; K}$ if $1 \le p \le q \le \infty$ and $\theta \in (0,1)$. Hence,  $W^{s,p} (\Rd) \subset B^{s}_{p,\infty} (\Rd) $ can
	also
	be deduced in this way.
	Embeddings between Besov spaces are discussed at length in \cite{Bennet1988}.
	Actually, through the mentioned embeddings, we can retrieve  \eqref{eq:fractional taylor less 1} but this time with an unknown constant $C(n,p,s)$ and depending on the $W^{s,p}$ norm (not the Gagliardo seminorm). To the best of our knowledge, it seems that \eqref{eq:fractional taylor less 1} is not widely-known.
	In practice,
	many authors use weaker, more difficult and less powerful properties of the fractional Sobolev spaces. We provide a clear and direct statement and proof of \eqref{eq:fractional taylor less 1} with  explicit constants.

\section{Applications}\label{sec.app}

We devote this section  to present several direct applications of the above results which motivated our investigation.

\subsection{Rellich–Kondrachov. Compact embeddings of  $W^{s,p}(\R^n)$}
\label{sec:Rellic-Kondrachov}
First, we prove the compact embedding of $W^{s,p}(\R^n)$ more directly
and with further generality
than in previous literature (cf. \cite[Section 7]{DiNezza2012}).
\normalcolor
\begin{theorem}[Fractional Rellich–Kondrachov]
	\label{thm:Rellich-Kondrachov}
	Let $p \in [1,\infty)$ and $\Omega \subset \Rd$ be measurable with finite measure and $K \subset \Rd$ compact.
	For $s \in (0,1)$
	We have the following:
	\begin{enumerate}
	
		\item\label{it:Rellich 2}  If $sp < n$ then $W^{s,p} (\Rd)$ is compactly  embedded
		in $L^q (\Omega)$ for $q < p^\star (n,s) \defeq  \tfrac{np}{n-sp}$.
		
		\item\label{it:Rellich 3}  If $sp =  n$ then $W^{s,p} (\Rd)$ is compactly  embedded
		in $L^q (\Omega)$ for $q < \infty$.
		
		\item\label{it:Rellich 4} If $sp >  n$ then $W^{s,p} (\Rd)$ is compactly  embedded
		in $C^\beta (K)$ for any  with $\beta < s- \tfrac n p$.
		
\setcounter{countersaver}{\value{enumi}}
	\end{enumerate}
	Furthermore, if \normalcolor $\Omega$ is a bounded with Lipschitz boundary, then
	\begin{enumerate}
\setcounter{enumi}{\value{countersaver}}
		\item\label{it:Rellich 5} If $s > t $ then $W^{s,p} (\Rd)$  is compactly  embedded in $W^ {t,p}  (\Omega)$.
	\end{enumerate}
	
\end{theorem}

More general results in which $W^{t,p}$ is replaced by $W^{t,q}$ for different values of $q$ can be obtained in a similar fashion.
We recall the fractional Gagliardo-Nirenberg interpolation inequality that can be found in \cite{Brezis2018}: if $\Omega$ is a standard domain (i.e. $\Rd$, a half-space or it is bounded with Lipschitz boundary) then
\begin{equation}
\label{eq:Gagliardo-Nirenberg interpolation}
\|  f \|_{W^{s,p} (\Omega)} \le C \| f \|_{W^{s_1,p_1} (\Omega)}^\theta \| f \|_{W^{s_2,p_2} (\Omega)}^{1-\theta}, \quad \theta \in (0,1), \quad s = \theta s_1 + (1-\theta) s_2 ,\quad \tfrac 1 p = \tfrac \theta {p_1} + \tfrac{1-\theta}{p_2}.
\end{equation}
as long at it fails that: $s_2$ is an integer $\ge 1$ and $p_2 = 1$ and $s_2 - s_1 \le 1 - \frac 1 {p_1}$.

The fractional version of the Gagliardo-Nirenberg-Sobolev inequality for $s \in (0,1)$ can be found in \cite{DiNezza2012}, where the compact embedding in \Cref{it:Rellich 2} in \Cref{thm:Rellich-Kondrachov}   and the continuous are proved for the rest of the cases. In \cite[Theorem 6.5]{DiNezza2012} for $sp <n$
\begin{equation}
\label{eq:GNS sp < n}
\| u \|_{L^{p^\star}(\Rd)} \le C[u]_{W^{s,p} (\Rd)}, \qquad \text{where } p^\star = p ^\star (n,s) ,
\end{equation}
for $sp = n$ in \cite[Theorem 6.9]{DiNezza2012}
\begin{equation*}
\| u \|_{L^{q}(\Rd)} \le C
\| u \|_{W^{s,p} (\Rd)}, \qquad \text{for all } q \in [p,+\infty),
\end{equation*}
and if $sp > n$ we have in \cite[Theorem 8.2]{DiNezza2012} that
\begin{equation}
\label{eq:GNS sp > n}
	\| u \|_{C^{0,\alpha} (\Rd)}  \le C \| u \|_{W^{s,p}(\Rd)},  \qquad \text{for  all } \alpha \le  s - \tfrac n p.
\end{equation}
For $s \in (0,1]$ and $sp > n$ it also worth mentioning that the inequality above holds in terms of the semi-norms in what it typically known as Morrey's inequality
\begin{equation*}
	|u(x) - u(y)| \le C [u]_{W^{s,p}(\Rd) } |x-y|^{s-\frac n p}.
\end{equation*}
The proof for $s = 1$ is classical and for $s \in (0,1)$, it may be found in \cite[Proposition 14.40 and Corollary 14.28]{Leoni2009Sobolev}
using Besov spaces
(see also \cite{Brasco2014} for a more direct proof).
For the sake of completeness, for $s \in (0,1)$ and $p \in [1,\infty)$ we recall the fractional Poincaré inequality (see \cite[Proposition 2.1]{Drelichman2018})
\begin{equation*}
\left \| u - \frac{1}{|\Omega|} \int_\Omega u (x) \diff x \right \|_{L^p (\Omega)} \le \frac{\textrm{diam} (\Omega)^{s + \frac n p} }{|\Omega|^{\frac 1 p}} \frac{ [u]_{W^{s,p} (\Rd)} }{ s^{\frac 1 p } (1-s)^{\frac 1 p} } .
\end{equation*}
We recall a particular form of the Fréchet-Kolmogorov Theorem
(see, e.g., Theorem 4.26 in \cite{Brezis:2010})
\begin{theorem}
	\label{thm:Frechet-Kolmogorov Brezis}
	Let $\mathcal F$ be a bounded set in $L^p (\Rd)$
	for $p \in [1,\infty)$
	and assume
	\begin{equation}
	\label{eq:FK uniform translations}
	\text{for every $\ee >0$, there exists $\delta > 0$ such that }   \| f (\cdot + y) - f \|_{L^p (\mathbb R^n)} \le \varepsilon ,  \forall |y| < \delta \text{ and } f \in \mathcal F.
	\end{equation}
	Then, for every $\Omega \subset \Rd$ measurable of finite measure the closure of $\mathcal F|_\Omega$ is compact in $L^p (\Omega)$.
\end{theorem}

With this result, we can proceed to the proof of \Cref{thm:Rellich-Kondrachov}\normalcolor.
Once \eqref{eq:fractional taylor less 1} is established, we rely on standard arguments. For similar proofs of \Cref{it:Rellich 2} see \cite[Theorem 6.5]{DiNezza2012} or \cite[Theorem 2.7]{Brasco2014} which follow the usual argument, see e.g. \cite[Theorem 9.16]{Brezis:2010}).
\begin{proof}[Proof of \Cref{thm:Rellich-Kondrachov}]
First, let us prove the the result for $q = p$.
\normalcolor
We take as $\mathcal F$  any bounded set of $W^{s,p} (\Rd)$.
Naturally, it is bounded in $L^p (\Rd)$. By \Cref{thm:main} we have that
	\begin{equation*}
		\| f(\cdot + y ) - f \|_{L^p (\Rd)} \le C \| f \|_{W^{s,p} (\Rd)} |y|^s \le  \tilde{C} |y|^s \quad  \textup{with} \quad \tilde{C}= C\sup_{f\in \mathcal{F}} \{\|f\|_{W^{s,p}(\Rd)}\}
	\end{equation*}
	for any $f \in \mathcal F$. Given $\ee > 0$, we can choose  $\delta = (\ee / \tilde{C})^{\frac 1 s}$ so that \eqref{eq:FK uniform translations} holds. Applying  \Cref{thm:Frechet-Kolmogorov Brezis} we conclude that $W^{s,p} (\Rd)$ is compactly embedded in $L^p (\Omega)$.

For the rest of the cases we prove sequential compactness.  Let $u_m \in W^{s,p} (\Rd)$ be a bounded sequence.

\ref{it:Rellich 2}) For $q \le p$ the result is trivial by the first part. Let $q \in [p,p^\star)$.
We prove sequential compactness. Let $u_m$ be a bounded sequence in $W^{s,p} (\Rd)$. By the previous part is has a subsequence, still denoted $u_{m}$, that converges strongly in $L^p (\Omega)$. By Hölder's inequality
\begin{equation*}
	\| u_{m} - u_k \|_{L^q (\Omega)} \le  \| u_m - u_k \|_{L^p (\Omega)}^{1-\theta} \|u_m - u_k \|_{L^{p^\star}  (\Omega)}^\theta \le 2 \| u_m - u_k \|_{L^p (\Omega)}^{1-\theta} \sup_\ell  \|u_\ell  \|_{L^{p^\star}  (\Omega)}^\theta
\end{equation*}
Hence, due to the continuous embedding \eqref{eq:GNS sp < n} and the fact that $u_m$ is a Cauchy sequence in $L^p(\Omega)$, we have that $u_m$ is a Cauchy sequence in $L^{q} (\Omega)$. Hence, it converges strongly in $L^q (\Omega)$.

\ref{it:Rellich 3}) The same argument applies.

\ref{it:Rellich 4}) We apply \eqref{eq:GNS sp > n} and the Ascoli-Arzelá theorem.

\ref{it:Rellich 5})  In this case, there exists a subsequence, still denote $u_m$ such that $u_m \to u$ in $L^p (\Rd)$.
We can apply  \eqref{eq:Gagliardo-Nirenberg interpolation} with $s_1 = s, s_2 = 0, \theta = t/s, p_0 = p_1 = p$ and write
\begin{align*}
\| u_m - u_n \|_{W^{t,p} (\Omega) } & \le C \| u_m - u_n \|_{W^{s,p} (\Rd) }^\theta  \| u_m - u_n \|_{L^p (\Omega) }^{1-\theta} \le C  \| u_m - u_n \|_{L^p (\Omega) }^{1-\theta}.
\end{align*}
Hence $u_m$ is Cauchy in $W^{t,p} (\Omega)$.
\end{proof}
\begin{remark}
	A natural question is whether these embeddings are  also compact into $L^p(\Rd)$. As in the integer case, this is not true.
	A very simple counterexample is the following. Let $u \in C_c^\infty (\Rd) \setminus \{0\}$ and consider the sequence $u_k (x) = u (x - k e_1)$ where $k \in \mathbb N$. Clearly $\| u_k \|_{W^{s,p} (\Rd)} = \| u \|_{W^{s,p} (\Rd)} $ so it is bounded. The weak limit of $u_k$ in $L^p (\Rd)$ is clearly $0$. Hence, the sequence cannot have an $L^p$-strongly convergence subsequence.
	
	The Fréchet-Kolmogorov Theorem states that a family $\mathcal F$ is relatively compact in $L^p (\Rd)$ if and only if it is bounded in $L^p(\R^n)$ and the following hold:
	\begin{enumerate}
		\item The tails are uniformly controlled, i.e. for every $\ee >0$ there exists $R$ such that
		\begin{equation*}
		\label{eq:FK uniform tails}
		\int_{|x| > R} |f|^p \le \ee , \qquad \forall f \in \mathcal F
		\end{equation*}
		\item The translations are uniformly controlled, i.e. \eqref{eq:FK uniform translations}.
	\end{enumerate}
	The boundedness of the tails cannot be uniformly controlled by the $W^{s,p}$ norm (as in the counterexample), and hence the weaker form \Cref{thm:Frechet-Kolmogorov Brezis} is more convenient.
\end{remark}

The reader may find in
\cite{Dao2018} and
 \cite{Dao+Diaz2019} an interesting proof of the continuous embeddings between fractional Sobolev spaces.

\medskip

\subsection{Rates of convergence of convolutions for $W^{s,p} (\Rd)$ functions}

A well-know result says that if $u \in L^p (\Rd)$ and the $\rho_\ee$ form a suitable family of mollifiers, then $\rho_\ee * u \to u$ in $L^p (\Rd)$. This fact is used in countless theoretical results (for example the proof of \Cref{thm:Frechet-Kolmogorov Brezis}). However, for certain applications it is relevant to give precise rates of convergence of these convolutions. Results for $W^{1,p} (\Rd)$ and $W^{2,p} (\Rd)$ are well-known.
The argument is also well known in some fractional Besov spaces.
\normalcolor

As a second application of our main result, we recover  fractional \normalcolor rates of convergence of convolutions,  up to a quadratic order. \normalcolor
Notice that the result is better for even mollifiers.
\begin{theorem}
	\label{thm:rate of convergence of convolutions}
Let  $\rho: \Rd \to \mathbb R$ is such that $\int_{\Rd} \rho (x) \diff x = 1$ \ and \ $\supp \rho = B_1$, and let $\rho_\ee (x) = \ee^{-n} \rho(x / \varepsilon)$. 	There exists $C = C(n)$ such that,
for every $u \in W^{s,p} (\Rd)$ with $p \in [\normalcolor1,\infty]\normalcolor$ and $s \in [0,1]\normalcolor$,  we have
		\begin{equation*}
		\| \rho_\ee * u - u \|_{L^p (\Rd)} \le \, C
			[ u ]_{W^{s,p} (\Rd)}
		\normalcolor
		\ee^{s}.
		\end{equation*}
	Furthermore,  if $\rho$ is even (i.e. $\rho (y) = \rho(-y)$) the result holds for $s \in [0,2]\normalcolor$.
\end{theorem}

\begin{proof}
	Consider first the case
	$s \in \bc[\normalcolor0,1]$.
	Then, by Jensen's inequality and Theorem  \ref{thm:main} we have
	\begin{equation}\label{eq:mollif}
	\begin{split}
	\|\rho_\veps*u-u\|_{L^{p}(\Rd)}^p&= \int_{\Rd} \left|\int_{B_\veps} \left(u(x+y)-u(x)\right)\rho_\veps(y)\dd y\right|^p \dd x\\
	& \leq \int_{B_\veps
	} \rho_\veps(y) \int_{\Rd}  \left|u(x+y)-u(x)\right|^p \dd x \dd y\\
	&\leq \sup_{|y|\leq \veps} \|u(\cdot+y)-u\|_{L^p(\Rd)}^p\leq c^p \veps^{s p}
	[u]_{W^{s ,p}(\Rd)}^p.
	\end{split}
	\end{equation}

	Now, let $s\in(1,2]$. Since $\rho$ is even we have that
	\begin{equation*}
	\int_{\Rd}  \nabla u (x) \cdot y \rho_\ee (y) \diff y =  \nabla u (x) \cdot	\int_{\Rd}  y \rho_\ee (y) \diff y = 0.
	\end{equation*}
	Hence, we can 	introduce the extra term $(\nabla u(x)\cdot y)\rho_\veps(y)$ in the first inequality of  \eqref{eq:mollif}. Again, Jensen's inequality implies,
	\[
	\begin{split}
	\|\rho_\veps*u-u\|_{L^{p}(\Rd)}^p&= \int_{\Rd} \left|\int_{B_\veps} \left(u(x+y)-u(x)-\nabla u(x) \cdot y\right)\rho_\veps(y)\dd y\right|^p \dd x\\
	& \leq \sup_{|y|\leq \veps} \|u(\cdot+y)-P_1u(\cdot,y)\|_{L^p(\Rd)}^p.\\
	\end{split}
	\]
	which allows us to conclude the desired result by using
	\Cref{thm:fractional taylor} with $k=1$ and $t=s-1$.  The limit cases are obtained in a similar way.\normalcolor
\end{proof}

\begin{remark}
	The conclusion of \Cref{thm:rate of convergence of convolutions} also holds when
	$\rho \in L^1 (\Rd)$
	is such that maybe
	$\|\rho\|_{L^1(\Rd)} \ne 1$, $\supp \rho \subset B_R$
	and it
	changes signs,
	and the result reads
	\begin{equation*}
	\left \|\rho_\veps*u -  u \int_{\Rd} \rho(x) \diff x  \right \|_{L^{p}(\Rd)} \le C  \| \rho \|_{L^1 (\Rd)}
	[u]_{W^{s,p}(\Rd)}
	\normalcolor \veps^s\,,
	\end{equation*}
	where $C$ also depends  on $R$.
\end{remark}

\subsection{Rates of convergence of  finite differences \normalcolor  for $W^{s,p} (\Rd)$ functions}

One of the most classical numerical methods to solve differential equations is given by finite differences. This method approximates the derivatives by pointwise evaluations  of the function itself. \normalcolor
The error of this approximation is sometimes referred by the name of   \emph{consistency of the discretization}. Estimates for this error in terms of integer derivatives are easily obtained from the Taylor expansion, and hence  such a \normalcolor theory is well-known  for \normalcolor  $W^{k,p}(\Rd)$ with $k\in\N$.

\medskip

In the recent literature, there has been significant interest in the study of PDEs of non-local or fractional type, in which the solution usually only lies in a fractional Sobolev class of regularity $W^{s,p}$. This leads to be interested in the study of the consistency of this discrete derivatives in the fractional Sobolev setting.
It has been shown (see, e.g. \cite{Ci-etal18, dTEnJa18}) that numerical  discretizations \normalcolor for this kind of fractional problems can be  constructed \normalcolor as fractional powers of this local finite-difference approximations.
 We obtain \normalcolor the following fractional consistency estimates:
\begin{theorem}\label{cor:localdisc}
	Let $u\in W^{s,p}(\Rd)$ and $p\in(1,\infty)$. We have that
	\begin{enumerate}[(a)]
		\item\label{item-cor:localdisc-a} If $s\in (1,2]$ then
		\begin{equation*}
		\left \| \frac{u(\cdot+e_ih)-u}{h}  - \frac{\partial u }{\partial x_i} \right  \|_{L^p(\Rd)}\leq c
			[u]_{W^{s,p}(\Rd)}
		h^{s-1}.
		\end{equation*}
		\item\label{item-cor:localdisc-b} If $s \in (2,4]$ then
		\begin{equation*}
		\left \|\frac{u(\cdot+e_ih)+u(\cdot-e_ih) -2u}{h^2}  - \frac{\partial^2 u }{\partial x_i^2} \right  \|_{L^p(\Rd)}\leq c
		[u]_{W^{s,p}(\Rd)}
		 h^{s-2}.
		\end{equation*}
	\end{enumerate}
 Here, $\{e_i\}_{i=1}^n$ denotes  the standard basis of $\Rd$.
\end{theorem}

\begin{proof}[Proof of \Cref{cor:localdisc}]
	We prove the result for $u \in C^\infty_c(\Rd)$, and the result holds in general by approximation.	
	First, we point out that \eqref{item-cor:localdisc-a} is precisely Theorem \ref{thm:fractional taylor} with $k=1$, $t=s-1$ and $y=h e_i$ and then dividing  by $h$.
	For  \eqref{item-cor:localdisc-b} we note that
	\begin{equation*}
	u(x\pm e_ih) = u(x) \pm h \frac{\partial u }{\partial x_i} (x) + \frac{h^2} 2  \frac{\partial^2 u }{\partial x_i^2} (x) \pm   \frac{h^3} 6 \frac{\partial^3 u }{\partial x_i^3} (x) +  R_\pm (x)
	\end{equation*}
	where, by \Cref{thm:fractional taylor}, we have $\| R_{\pm} (x)\|_{L^p (\Rd)} \le C h^{s}
	[u]_{W^{s,p}(\Rd)}
	 $ for $s \in (2, 4]$. Hence
	\begin{align*}
	\frac{u(x+e_ih)+u(x - e_i h)-2u(x)}{h^2} - \frac{\partial^2 u }{\partial x_i^2} (x)  &= \frac{1}{h^2} \left( R_{+} (x) - R_- (x) \right).
	\end{align*}
	This completes the proof.
\end{proof}
\begin{remark}
	Applying these techniques, one can recover similar estimates for any of the classical finite-difference approximations.
\end{remark}

\section{Proof of the main results}
\label{sec.thm}

The main technique we will use in the proof of \Cref{thm:main} is
the theory of interpolation spaces.
Interpolation techniques are delicate and are sometimes used carelessly. In order to be very precise, we will introduce several auxiliary definitions.
For $s >0$, we can define $W^{s,p} (\Rd)$ as the closure of $C_c^\infty (\Rd)$ in $L^p(\Rd)$  with respect to the $\| \cdot \|_{W^{s,p} (\Rd)}$ norm, or equivalently
\begin{equation*}
	W^{s,p} (\Rd) = \{  u \in L^p ( \Rd ) :  \| u \|_{W^{s,p} (\Rd)} < +\infty  \}.
\end{equation*}

The $K$-interpolation space  for $s \in (0,1)$ and $p\in [1,\infty)$  of two spaces $X_0$ and $X_1$  is  defined as
the elements in $X_0 + X_1$ which have finite
(normalised) $K$-interpolation norm
\begin{equation}
\label{eq:K interpolation}
\vertiii{u}_{(X_0, X_1)_{s,p; K}} \defeq \left( s (1-s) \int_0^{+\infty} \left( \frac{K(t,u)}{t^{s}} \right)^p \frac{\diff t}{t} \right)^{\frac 1 p},
\end{equation}
where
\begin{equation*}
K(t,u) = K(t,u,X_0, X_1) \vcentcolon= \inf \left\{   \|u_0\|_{X_0} +  t \| u_1 \|_{X_1} \, : \, u=u_0+u_1, \, u_0\in X_0, \, u_1\in X_1 \right\} .
\end{equation*}
In order for this definition to make sense, $X_0 \cap X_1$ must be nontrivial and $X_0 + X_1$ must be a topological vector space. It is then said that $\{X_0,X_1\}$ is an interpolation pair.
Notice that, for each $s$, removing $s(1-s)$ in \eqref{eq:K interpolation} gives an equivalent norm. In fact, most books (e.g. \cite{Adams+Fournier:2003sobolev+spaces,Bennet1988}) do not include this constant. However, Milman points out in  \cite{Milman2005} that the behaviour mentioned in \Cref{rem:normalised Gagliardo-Nirenberg} occurs in general, and this constant $s(1-s)$ allows to pass to the limits as $s \to 0,1$.

A well known interpolation result asserts that
\begin{equation}
\label{eq:sobolev as interpolation}
W^{s,p} (\Rd) = (L^p (\Rd), W^{1,p} (\Rd))_{s,p;K}
\end{equation}
via the $K$-interpolation
(see \Cref{rem:Tartar}).
A very complete but rather technical presentation of this kind of results can be found in \cite{Amann2000}.
Let us define
for $u \in C_c^\infty (\Rd)$ the (normalised) interpolation-like semi-norm
\begin{equation*}
\vertiii{u}_{s,p} \defeq \left( s (1-s) \int_0^{+\infty} \left( \frac{\overline K(t,u)}{t^{s}} \right)^p \frac{\diff t}{t} \right)^{\frac 1 p}
\end{equation*}
where
\begin{equation}\label{eq:DefKline}
	\overline K (t,u) = \inf_{u_1 \in C_c^\infty (\Rd)} \left\{   \|u - u_1\|_{L^p (\Rd)} +  t \|\nabla  u_1 \|_{L^p(\Rd)} \right\}.
\end{equation}
In \cite{Brasco2019} the authors prove that, for $u \in C_c^\infty (\Rd)$, this norm is equivalent to the (normalised) Gagliardo seminorm
\begin{equation}
\label{eq:Brasco}
\frac 1 C [ u ]_{W^{s,p} (\Rd)} \le \vertiii{u}_{s,p}  \le C [ u ]_{W^{s,p} (\Rd)},
\end{equation}
where $C$ depends only on $n$ and $p$.

 \begin{remark}
 	\label{rem:homogeneous Sobolev}
Some authors define a space $\dot W^{s,p} (\Rd)$ as the completion of $C_c^\infty(\R^n)$ with respect the semi-norm $[\cdot]_{W^{s,p}(\R^n)}$. In these terms, the results in \cite{Brasco2019} would imply
\[
\dot W^{s,p} (\Rd) \simeq (L^p (\Rd), \dot W^{1,p} (\Rd))_{s,p;K}.
\]
However, this introduces some functional analysis difficulties that we want to avoid here.
	We point out that the completion is a delicate process.
		Let
		$s \in (0,1]$ and	
		$sp > n$.
		Let $\varphi \in C_c^\infty (\Rd)$
		with value $1$ in $B_1$ and $0$ outside $B_2$ and $\varphi_m (x) = \varphi (x/m)$. By a simple change in variable, it is easy to see that $[ \varphi_m ] _{W^{s,p} (\Rd)} = m^{\frac n p - s}  [ \varphi ] _{W^{s,p} (\Rd)}$. We thus get $\varphi_m \to 1$ uniformly over compacts and  $[\varphi_m]_{W^{s,p} (\Rd)} \to 0$. Therefore, in the completion the zero function is equivalent to all constants.	
\end{remark}

\normalcolor

A key step in the proof is to use
\eqref{eq:Brasco}
 together
with
the fact that
$X_0 = (X_0, X_0)_{s,p;K}$ and the explicit value of (normalised) interpolation norm is given by
\begin{equation}
\label{eq:interpolated Lp norm}
\| u \|_{X_0} = p^{\frac 1 p} \vertiii{u}_{(X_0,X_0)_{s,p;K}}.
\end{equation}
(see \cite[Section 2.1]{Milman2005})
We now proceed to the proof of the main result.

\begin{proof}[Proof of \Cref{thm:main}]

	\noindent \textbf{Case 1:   $p = \infty$  {\rm and} $s\in[0,1]$.} This  is  a classical due to the definition of $W^{s,\infty}(\R^n)$.
	
	From here and in the following cases, we will assume that $u \in C_c^\infty (\Rd)$ and the results  in $W^{s,p} (\Rd)$ hold by approximation.
	
	\noindent \textbf{Case 2: $p \in (1,\infty)$ {\rm and} $s \in \{0,1\}$.}
	It is well known  (see for example \cite[Proposition 9.3]{Brezis:2010} for a clear proof) \normalcolor that, for $u \in   C_c^\infty(\Rd)$, we have
	\begin{equation}\label{eq:clasicaltrans}
	\|  u(\cdot+y) - u \|_{L^p (\Rd)} \le\| \nabla u \|_{L^p (\Rd)}  |y| \quad   \textup{and} \quad \|  u(\cdot+y) - u \|_{L^p (\Rd)} \le 2 \| u \|_{L^p (\Rd)}.
	\end{equation}
	\noindent \textbf{Case 3: \normalcolor $p \in (1,\infty)$ {\rm and}  $s \in (0,1)$.}
  For $y$ fixed, let $T : u \mapsto  u(\cdot+y) - u$.
	Then $Tu \in C_c^\infty (\Rd)$.  Now we partially reproduce the proof of \cite[Theorem 7.23]{Adams+Fournier:2003sobolev+spaces} (which we state later as \Cref{thm:exact interpolation}).  We write
	\begin{align*}
		K(t,Tu, L^p (\Rd), L^p(\Rd)) &= \inf \left\{   \| T u_0\|_{L^p(\Rd)} +  t \| T u_1 \|_{L^p (\Rd)} \, : \, u=u_0+u_1, \, u_i \in L^p (\Rd)\right\} \\
		&= \inf \left\{   \|Tu - Tu_1\|_{L^p(\Rd)} +  t \| Tu_1 \|_{L^p (\Rd)} \, : \, u_1 \in L^p (\Rd)\right\} \end{align*}

Actually, we can consider the infimum over the functions $u_1\in C_c^\infty(\Rd)$, since for a minimising sequence in $L^p(\Rd)$ we can construct another minimizing sequence in $C_c^\infty(\Rd)$ (for example by mollification). This observation together with  \eqref{eq:clasicaltrans} and the definition of $\overline{K}$ given in \eqref{eq:DefKline} lead to
		\begin{align*}
		K(t,Tu, L^p (\Rd), L^p(\Rd))&= \inf \left\{   \|Tu - Tu_1\|_{L^p(\Rd)} +  t \| Tu_1 \|_{L^p (\Rd)} \, : \, u_1 \in C_c^\infty (\Rd)\right\} \\
		&\le \inf \left\{  2 \|u - u_1\|_{L^p(\Rd)} +  t |y| \| \nabla u_1 \|_{L^p (\Rd)} \, : \, u_1 \in C_c^\infty (\Rd)\right\} \\
		&= 2  \inf \left\{  \|u -u_1\|_{L^p(\Rd)} +   \textstyle\frac{t |y|}2 \| \nabla u_1 \|_{L^p (\Rd)} \, : \, u_1 \in C_c^\infty (\Rd)\right\} \\
		&= 2 \overline K \left(\textstyle\frac{t |y|}2 , u \right).
	\end{align*}
	Inserting the previous estimate in the definition of $K$-interpolant norm given in \eqref{eq:K interpolation} and changing variables, we get
	\begin{align*}
		\vertiii{ u(\cdot+y) - u}_{(L^p (\Rd), L^p(\Rd))_{s,p;K}} &\le  \left( s (1-s) \int_0^{+\infty} \left( 2\overline K\left(\frac{t |y|}2 , u \right) \frac{1}{t^{s}} \right)^p \frac{\diff t}{t} \right)^{\frac 1 p}\\
	&\le 2^{1-s} |y|^s \left( s (1-s) \int_0^{+\infty} \left( \frac{\overline K\left(t , u \right)}{t^{s}} \right)^p \frac{\diff t}{t} \right)^{\frac 1 p}\\
	&=2^{1-s}   \vertiii{u}_{s,p} |y|^s.
	\end{align*}
	\normalcolor
	This,  together with identity \eqref{eq:interpolated Lp norm}, implies
	\begin{equation*}
	\begin{split}
	\| u(\cdot+y)- u \|_{L^p (\Rd)} &= p^{\frac{1}{p}} \vertiii{u(\cdot+y)- u}_{(L^p(\Rd), L^p(\Rd))_{s,p;K}} \\
	&\leq p^{\frac 1 p}2^{1-s}  \vertiii{u}_{s,p} |y|^s%
	\end{split}
	\end{equation*}
	From \cite[Proposition 4.5]{Brasco2019} we recover \eqref{eq:Brasco} with an explicit constant. More precisely, 	\begin{equation*}
	\vertiii{u}_{s,p} \le \frac{2n(n+1)}{(n\omega_n)^{\frac 1 p}} [u]_{W^{s,p} (\Rd)}
	\end{equation*}	
	Hence
	\begin{equation}
	\label{eq:intermediate}
		\| u(\cdot+y)- u \|_{L^p (\Rd)} \le \frac{2^{2-s} n(n+1)p^{\frac 1 p}}{(n\omega_n)^{\frac 1 p}}[u]_{W^{s,p} (\Rd)} |y|^s
	\end{equation}
		which is precisely \eqref{eq:fractional taylor less 1}.  We can estimate the constant uniformly in $s\in[0,1]$ and $p\in[1,\infty]$ by joining  \eqref{eq:clasicaltrans} and \eqref{eq:intermediate} to get
	\begin{equation}\label{eq:constantupperbound}
		\max \left \{ 2 ,  \frac{2^{2-s} n(n+1)p^{\frac 1 p}}{(n\omega_n)^{\frac 1 p}}  \right \} \le 4 n(n+1) e ^{\frac 1{e \omega_n n}} \definedby C(n),
	\end{equation}
	
	\medskip
	
	\noindent \textbf{Case 4:   $p = 1$ {\rm and}  $s \in [0,1]$.}  From the explicit values above, there is a finite limit as $p \to 1$. By the explicit formula of the Gagliardo seminorm\normalcolor, it is clear that for $u \in C_c^\infty (\Rd)$ we have $[u]_{W^{s,p} (\Rd)} \to [u]_{W^{s,1} (\Rd)}$ as $p \searrow 1$. %
\end{proof}

\begin{remark}\label{rem:contantn}
Note that in \eqref{eq:constantupperbound}, we have  $n\omega_n  \to 0$ as $n\to\infty$,  and thus $C(n)$ diverges as $n\to\infty$, which does not allow us to get a uniform bound in $n$.
\end{remark}
 We are now ready to prove  our result regarding higher-order Sobolev spaces.

\begin{proof}[Proof of  \Cref{thm:fractional taylor}]

 	Let $k = 1$. For $u \in  \mathcal C_c^2 (\Rd)$ we write the Taylor expansion
 	\begin{equation*}
 	u(x+y) = u(x) + \int_0^{1} \nabla u (x + t y) \cdot y \diff t .
 	\end{equation*}
 	Therefore,
 	\begin{equation*}
 	 u (x+y\normalcolor) - (u(x) + \nabla u (x) \cdot y) = \int_0^1 \left( \nabla u (x + ty) -  \nabla u(x) \right)  \cdot y \diff t\\.
 	\end{equation*}
 Integrating over $\Rd$, \normalcolor applying Jensen's inequality and Fubini's theorem we have
 	\begin{align*}
 	\int_{\Rd} | u (x+y\normalcolor) - (u(x) + \nabla u (x) \cdot y) |^p \diff x &\le \int_{ \Rd} \int_0^1  |\nabla u (x + ty) -  \nabla u(x) |^p |y|^p\diff t  \diff x\\
 	&\le |y|^p  \int_0^1  \|  \nabla u (\cdot+ty)\normalcolor - \nabla u \|_{L^p (\Rd)}^p \diff t\\
 	&= |y|^p \sup_{t \in [0,1]}   \| \nabla u(\cdot+ty) - \nabla u \|_{L^p (\Rd)}^p\,.
 	\end{align*}
 	Since by definition of the fractional Sobolev spaces, $\nabla u \in W^{s-1,p} (\Rd)^n $ we apply \Cref{thm:main} to recover
 	\begin{equation*}
 	\int_{ \Rd} | u(x+ y)\normalcolor - (u(x) - \nabla u (x) \cdot y) |^p \diff x \le C |y|^p |y|^{p(s-1)} \| \nabla u \|_{W^{s-1,p} (\Rd)}^p.
 	\end{equation*}
  	The  cases when $k > 1$ are  proved by induction and an analogous argument.	
\end{proof}

For the reader's convenience we recall
an
interpolation result that can be found in \cite{Adams+Fournier:2003sobolev+spaces},
since we apply part of its proof above.
\begin{theorem}[Theorem 7.23(a) in \cite{Adams+Fournier:2003sobolev+spaces}]
	\label{thm:exact interpolation}
	Let $\{X_0,X_1\}$ and $\{Y_0, Y_1\}$ be two interpolation pairs. Let, for $\theta \in (0,1)$ and $q \in [1,\infty)$
	\begin{equation*}
	X_\theta = (X_0, X_1)_{\theta, q;K}, \qquad Y_\theta = (Y_0, Y_1)_{\theta, q;K }
	\end{equation*}
	with their $K$-interpolation norms.
	Then, for any $T: X_0 + X_1 \to Y_0 + Y_1$ such that it is linearly continuous $T: X_i \to Y_i$
	we have that
	\begin{equation*}
	\| T \|_{ \mathcal L (  X_\theta, Y_\theta    )    } \le 	\| T \|_{ \mathcal L (  X_0, Y_0    )    }^{1-\theta} \| T \|_{ \mathcal L (  X_1, Y_1    )    }^{\theta} .
	\end{equation*}
\end{theorem}
\normalcolor

\section{Comments and open problems}
\label{sec:comments}
\begin{enumerate}

	\item Some general Rellich-Kondrachov embeddings for different $q$ can also be obtained by interpolation. 	

	\item
	By extending the argument in \Cref{thm:Rellich-Kondrachov}, it is possible to prove the general Sobolev compact embedding theorem, i.e. possible to prove that $0 \le t < s$ the space $W^{s,p} (\Rd)$ is compactly embedded in $W^{t,q} (\Omega)$ if $\frac{1}p - \frac s n < \frac 1 q - \frac t n$.
	
	\item In \cite{Brezis} the authors fix a ``defect'' of the Gagliardo at $s =1$ by replacing the integration by a norm in a Marcinkiewicz space. It is interesting to know if the results in this paper could adapt to those norms.
	
	\item Our results apply to domains $\Omega$, as long as there is an extension operator $E : W^{s,p} (\Omega) \to W^{s,p} (\Rd)$.
	A nice discussion of the fractional Sobolev in bounded domains and the existence of this extensions can be found, e.g., in \cite{DiNezza2012}.

	\item The convergence of the convolution under suitable hypotheses for $\rho_\veps$ in \Cref{thm:rate of convergence of convolutions} for $s > 2$ is left open.
	This kind of result have been studied for $p = 2$ by Fourier transform.

	\item  Discretizations of nonlocal operators like the fractional Laplacian $(-\Delta)^s$ based on Lagrange interpolants (see \cite[Section 4.4]{dTEnJa18}) can be shown to have a consistency error in $L^p(\Rd)$ of order $h^{\sigma-2s}$ for $W^{\sigma,p}(\Rd)$. This topic will be thoroughly studied in the forthcoming paper \cite{dTGCVa20}.
	
	\item When $s=\{0,1\}$, the constant in \Cref{thm:main} does not depend on the dimension  either. We wonder if the dependence on $n$ can be dropped also when $s\in(0,1)$ (see \Cref{rem:contantn}). More generally, the calculation of the optimal constant for $s\in(0,1)$ and variable $p$ is an important question.

\end{enumerate}

\section*{Acknowledgements}
F. del Teso was partially supported by PGC2018-094522-B-I00 from the MICINN, of the Spanish Government.
The work of D. G\'omez-Castro  and J. L. V\'azquez were funded by  grant PGC2018-098440-B-I00 from  the  MICINN,  of the Spanish Government. J.~L.~V\'azquez is an Honorary Professor at Univ.\ Complutense de Madrid. The authors are thankful to L. Brasco for
information and useful comments.

\printbibliography

\end{document}